\documentclass[12pt]{article}
\usepackage{amsmath}
\usepackage{amssymb,latexsym}
\usepackage{amsthm,amscd}
\theoremstyle{plain}
\newtheorem{theorem}{Theorem}[section]
\newtheorem{lemma}[theorem]{Lemma}
\newtheorem{corollary}[theorem]{Corollary}
\newtheorem{proposition}[theorem]{Proposition}
\theoremstyle{remark}

\newtheorem{example}[theorem]{Example}
\begin{document}
\title{Flocks of Cones: Herds and Herd Spaces }
\author{William Cherowitzo}
\maketitle
\section{Introduction}

This is the first in a series of articles devoted to providing a
foundation for a theory of flocks of arbitrary cones in
$PG(3,q)$. The desire to have such a theory stems from a need to
better understand the very significant and applicable special
case of flocks of quadratic cones in $PG(3,q)$. Flocks of
quadratic cones have connections with several other geometrical
objects, including certain types of generalized quadrangles,
spreads, translation planes, hyperovals (in even characteristic),
ovoids, inversive planes and quasi-fibrations of hyperbolic
quadrics. This rich collection of interconnections is the basis
for the strong interest in such flocks. Recent work has shown that some of these connections can be made with other types of cones.  The author has attempted
incremental generalizations of flocks of quadratic cones
(\cite{WEC:98},\cite{WEC:98b}) and the similarity of the results
in these investigations indicated the existence of a more general
framework. However, this incremental approach leads to more and
more difficult algebraic considerations that ultimately make this
approach untenable. By jumping to the most general situation and
changing our point of view (as we will do in this series of
articles) we can transcend those algebraic difficulties and
hopefully gain a clearer perspective on the subject. This first paper lays out the fundamentals while later papers will examine special types of flocks.

    The reader should be warned that we have taken some liberties
with terminology (especially the terms ``flock'' and ``herd'') by
redefining objects in a more general context. This was deemed
necessary to avoid having to introduce a more cumbersome set of
terms. However, the use of modifiers will insure that the new
definitions are in agreement with the more commonly used ones in
the appropriate context. We should also mention that we have
restricted ourselves to the finite case only out of preference
for that setting. Infinite analogs of almost everything that
appears here do exist, but we shall leave it to others to develop
these.

\section{Cones and Flocks}

Let $\pi_0$ be a plane and $V$ a point not on $\pi_0$ in
$PG(3,q)$. Let $\mathcal{C}$ be any set of points in $\pi_0$
(including the empty set). A \emph{cone}, $\Sigma =
\Sigma(V,\mathcal{C})$ is the union of all points of $PG(3,q)$ on
the lines $VP$ where $P$ is a point of $\mathcal{C}$. $V$ is
called the \emph{vertex} and $\mathcal{C}$ is called the
\emph{carrier} of $\Sigma$. $\pi_0$ is the \emph{carrier plane}
and the lines $VP$ are the \emph{generators} of $\Sigma$. In the
event that $\mathcal{C} = \emptyset$ we call $\Sigma$ the
\emph{empty cone} and by convention consider it to consist of
only the point $V$.

    A \emph{flock of planes} in $PG(3,q)$ is any set of $q$
\emph{distinct} planes of $PG(3,q)$. As $q$ planes can not cover
all the points of $PG(3,q)$, there always are points of the space
which do not lie in any of the planes in a flock of planes. If
$\Sigma$ is a cone of $PG(3,q)$, then a flock of planes,
$\mathcal{F}$, is said to be a \emph{flock of }$\Sigma$ when the
vertex of $\Sigma$ lies in no plane of $\mathcal{F}$ and no two
planes of $\mathcal{F}$ intersect at a point of $\Sigma$. Any
flock of planes is a flock of a cone, possibly only the empty
cone. In general, however, a given flock of planes will be a flock
of several cones.  In the literature on flocks of quadratic
cones, the approach is always to consider a fixed quadratic cone
and study the flocks of that cone. We will change the viewpoint
and consider, for a fixed flock of planes, the various cones of
which it is a flock. In the sequel we shall refer to a flock of
planes simply as a \emph{flock} and it shall be understood that
it is always a flock of a cone, even if the cone is not
explicitly indicated.

    In order to provide an algebraic representation of a flock we
will need to introduce coordinates. The standard homogeneous
coordinates of a point in $PG(3,q)$ will be given by $\langle
x_0, x_1, x_2, x_3 \rangle$ with $\langle \cdots \rangle$
denoting the fact that we are dealing with an equivalence class.
When a specific element of this equivalence class is needed, it
will be denoted with parentheses. Thus $(a,b,c,d) \in \langle
a,b,c,d \rangle$ is a specific representative of the equivalence
class.

    Let $\mathcal{F}$ be a flock. We can introduce coordinates in
$PG(3,q)$ so that the plane $x_3 = 0$ is one of the planes of the
flock and the point $V = \langle 0,0,0,1 \rangle$ is not in any
plane of the flock. We parameterize the planes of $\mathcal{F}$
with the elements of $GF(q)$ in an arbitrary way except that we
will require that $0$ is the parameter assigned to the plane $x_3
= 0$. We can now describe the flock as, $\mathcal{F} = \{\pi_t
\mid t \in GF(q)\}$ with $\pi_0 \colon x_3 = 0$. Since $V$ is not
in any plane of $\mathcal{F}$, each of the planes of this flock
has an equation of the form $Ax_0 + Bx_1 + Cx_2 - x_3 = 0$.
Consider the points $P = \langle 1,0,0,0 \rangle, Q = \langle
0,1,0,0 \rangle$ and $R = \langle 0,0,1,0 \rangle$ of $\pi_0$.
The points other than $V$ on the lines $VP, VQ \text{ and }VR$
are given by $\langle 1,0,0,\lambda \rangle, \langle 0,1,0,\mu
\rangle$ and $\langle 0,0,1,\nu \rangle$ respectively, with
$\lambda, \mu, \nu$ varying in $GF(q)$. For each $t \in GF(q)$
the plane $\pi_t$ of the flock meets these lines at the points $(
1,0,0,\lambda_t ),( 0,1,0,\mu_t)$ and $( 0,0,1,\nu_t)$
respectively (note the use of specific representatives). We
define three functions $f,g,h \colon GF(q) \to GF(q)$ by $f(t) =
\lambda_t, g(t) = \mu_t \text{ and } h(t) = \nu_t$. These
functions describe the equations of the planes of the flock,
namely $\pi_t \colon f(t)x_0 + g(t)x_1 + h(t)x_2 - x_3 = 0$. The
functions $f,g \text{ and } h$ are called the \emph{coordinate
functions} of the flock. Note that the requirement on the
parameter $0$ means that $f(0) = g(0) = h(0) = 0$. If $f,g \text{
and } h$ are the coordinate functions of the flock $\mathcal{F}$
we shall write $\mathcal{F} = \mathcal{F}(f,g,h)$. We remark that
the coordinate functions of a flock depend on both the
parameterization of the flock and the procedure used to obtain
the functions from the homogeneous coordinates (i.e., the
selection of representatives  of these coordinates). We will have
occasion to change the parameterization of a flock but will never
change this standard procedure for obtaining the functions.

    All cones under consideration will have vertex $V = \langle
0,0,0,1 \rangle$ and we will consider the plane $\pi_0$ as the
carrier plane of the cone. Thus, a cone is determined when its
carrier $\mathcal{C}$, a point set in $\pi_0$, is specified.
Given a flock $\mathcal{F}$, there is a largest set
$\mathcal{C}_0$ of $\pi_0$ such that $\mathcal{F}$ is a flock of
the cone with carrier $\mathcal{C}_0$. This cone is called the
\emph{critical cone} of $\mathcal{F}$. If $\mathcal{C}$ is any
subset of the carrier of the critical cone of a flock
$\mathcal{F}$, then clearly $\mathcal{F}$ is also a flock of the
cone with carrier $\mathcal{C}$. Thus, determining the critical
cone of a flock implicitly determines all cones for which this
flock of planes is a flock.

    The critical cone of a flock may be fairly ``small''. Besides
the empty cone, we will consider cones whose carriers consist of
collinear points as being ``small''. Cones of this type are
called \emph{flat cones}. For the most part, we shall regard
flocks whose critical cones are flat as being uninteresting.

\section{Herds and Herd Spaces}
    Let $\mathcal{Z}$ denote the collection of all functions $f\colon
GF(q) \to GF(q)$ such that $f(0) = 0$. Note that each element of
$\mathcal{Z}$ can be expressed uniquely as a polynomial in one
variable of degree at most $q-1$ and that $\mathcal{Z}$ is a
vector space over $GF(q)$. We shall always consider the vectors of
$\mathcal{Z}$ as polynomials of this type and use
$\{t,t^2,\ldots,t^{q-1}\}$ as the standard basis of
$\mathcal{Z}$. Let $\mathcal{V} = \mathcal{V}(f, g, h)$ denote
the subspace of $\mathcal{Z}$ generated by the vectors $f,g
\mbox{ and } h$ of $\mathcal{Z}$. Also, let $\mathcal{U} = GF(q)
\times GF(q) \times GF(q)$ considered as a vector space over
$GF(q)$. Now, the map $\phi_{f,g,h} \colon \mathcal{U} \to
\mathcal{V}$ given by:
\begin{equation} \label{E:stmap}
\phi_{f,g,h}(a,b,c) = af + bg + ch
\end{equation}
is a non-trivial vector space homomorphism, provided that at least
one of the functions is not the constant function (which we shall
always assume to be the case). In the sequel we will suppress the
subscripts in the name of this homomorphism whenever this will
not lead to confusion. In the usual manner we may construct the
projective geometries
$\mathbf{P}(\mathcal{Z}),\mathbf{P}(\mathcal{U})\mbox{ and
}\mathbf{P}(\mathcal{V})$ from the vector spaces
$\mathcal{Z},\mathcal{U}\mbox{ and }\mathcal{V}$ respectively.
The projective space $\mathbf{P}(\mathcal{Z})$ is isomorphic to
$PG(q-2,q)$ and $\mathbf{P}(\mathcal{U})$ is just $PG(2,q)$ in
its usual representation. It is clear that $\phi$ induces a map
$\hat{\phi}\colon\mathbf{P}(\mathcal{U}) \to
\mathbf{P}(\mathcal{V})$ given by:
\begin{equation} \label{E:stprmap}
\hat{\phi}(\langle a,b,c \rangle) = \langle af + bg + ch \rangle.
\end{equation}
The set of ordered pairs that defines $\hat{\phi}$ (sometimes
called the graph of $\hat{\phi}$), which we denote by
$\Gamma(f,g,h)$, i.e.,
\begin{equation} \label{E:hdsp}
\Gamma(f,g,h) = \{ (\langle a,b,c \rangle,\langle af + bg + ch
\rangle) \mid \langle a,b,c \rangle \in \mathbf{P}(\mathcal{U})\}
\end{equation}
is called the \emph{Herd Space} of the functions $f,g \mbox{ and
} h$.

    In the important case that $\hat{\phi}$ is bijective,
$\mathbf{P}(\mathcal{V})$ is a projective plane and $\hat{\phi}$
is just the inverse of the standard coordinate function, when
$\{f, g, h\}$ is considered as an ordered basis of $\mathcal{V}$.
In this case the herd space can be thought of as a plane in
$\mathbf{P}(\mathcal{Z})$ which has been coordinatized in a
special way.

    For reasons that will become clear later, we say that a herd
space $\Gamma(f,g,h)$ is \emph{degenerate} if there exist
distinct elements $u,v \in GF(q)$ such that $f(u) = f(v), g(u) =
g(v) \text{ and } h(u) = h(v)$. Thus, if any of $f,g \text{ or } h$ is
a permutation (i.e., a permutation polynomial) then
$\Gamma(f,g,h)$ is non-degenerate.

    There are $(q-1)!$ elements of $\mathcal{Z}$ which are permutation
functions. As all non-zero scalar multiples of a permutation are
also permutations, there are $(q-2)!$ points of
$\mathbf{P}(\mathcal{Z})$ which are classes of permutation
functions and will be referred to as \emph{permutation points}.
Let $\mathcal{S}$ denote the set of all permutation points in
$\mathbf{P}(\mathcal{Z})$. A corollary to Hermite's condition for
permutation polynomials (see \cite{LED:01}) states that the degree
of a permutation polynomial over $GF(q)$ is either 1 or it does
not divide $q-1$. Thus, $\mathcal{S}$ lies in the hyperplane
$\langle t, t^2, \ldots,t^{q-2} \rangle$ of
$\mathbf{P}(\mathcal{Z})$.

    We now define a significant partial function of the herd space
$\Gamma(f,g,h)$, called the \emph{Herd Cover} of $f,g \mbox{ and
} h$, and denoted by $\mathcal{HC} = \mathcal{HC}(f,g,h)$, where
\begin{equation} \label{E:hdcv}
\mathcal{HC} = \{ (\langle a,b,c \rangle,\langle af+bg+ch \rangle)
\mid \langle af+bg+ch \rangle \in \mathcal{S}\}.
\end{equation}
Note that $\mathcal{HC}(f,g,h)$ may be (and often is, for
arbitrary $f,g \mbox{ and }h$) the empty set. Also note that, a
priori, there is no geometrical significance to a point being a
permutation point in $\mathbf{P}(\mathcal{V})$, as this is a
purely algebraic notion. However, as we shall see below, these
points gain an important geometric role in the context of flocks.

    For any given herd cover, $\mathcal{HC} =
\mathcal{HC}(f,g,h)$, the image of the standard projection map
onto the first coordinates, $\pi_1\colon\mathcal{HC} \to
\mathbf{P}(\mathcal{U})$, is called the \emph{point set of}
$\mathcal{HC}$, and denoted by $\mathbf{P}_{\mathcal{HC}}$. Any
map $\rho \colon\mathbf{P}(\mathcal{U}) \to \mathcal{U}$ such
that $\rho(\langle a,b,c \rangle) \in \langle a,b,c \rangle$ is
called a \emph{herd selection function}. As an aside we note that
in the infinite case the existence of these choice functions may
require the Axiom of Choice, but since we have restricted
ourselves to the finite case, this issue does not arise. Finally,
the set of representatives of the herd cover determined by the
herd selection function $\rho$, denoted by
$\rho\,\text{-}\mathcal{H}(f,g,h)$, and given by,
\begin{equation} \label{E:herd}
\rho\,\text{-}\mathcal{H}(f,g,h) = \{ (\rho(P),\phi(\rho(P)) \mid
P \in \mathbf{P}_{\mathcal{HC}}\},
\end{equation}
is called the $\rho\,$-\emph{Herd} of the herd cover
$\mathcal{HC}(f,g,h)$. Essentially, a herd is just a
representation of a herd cover where the representative of the
first coordinate is arbitrary and the representative of the
second coordinate depends on the first coordinate choice. It is
clear that there are several $\rho\,$-herds associated to a given
herd cover and that the herd cover is uniquely determined by any
of its $\rho\,$-herds.

Since we have given a formal definition of a $\rho\,$-herd and
freely admit that its notation is a bit cumbersome, we will use a
simpler alternative notation when it will not lead to confusion,
namely,
\begin{equation} \label{E:simpherd}
\rho\,\text{-}\mathcal{H}(f,g,h) = \{ f_P \mid P \in
\mathbf{P}_{\mathcal{HC}},f_P = \phi(\rho(P))\}.
\end{equation}
In this view, a $\rho\,$-herd is an indexed family of permutation
functions where the indexing set is $\mathbf{P}_{\mathcal{HC}}$.

    There are several useful choices for the herd selection function
$\rho$ and we shall discuss a few.  The \emph{standardized herd}
is the $\rho\,\text{-}\mathcal{H}(f,g,h)$ such that
\begin{equation} \label{E:stanherd}
 f_P =
 \begin{cases}
 \frac{a}{c}f + \frac{b}{c}g + h,     &\text{if $c \ne 0$;}  \\
 f + \frac{b}{a}g,                    &\text{if $c=0$ and $a \ne 0$;} \\
 g,                                   &\text{if $a = c = 0$.}
 \end{cases}
   \text{  , where } P = \langle a,b,c \rangle .
\end{equation}
The herd selection function for the standardized herd is the one
which selects the representatives of $\langle a,b,c \rangle$ of
the forms $(x,y,1), (1,m,0) \text{ or } (0,1,0)$. We also have the
\emph{alternate standardized herd} which is the
$\rho\,\text{-}\mathcal{H}(f,g,h)$ such that
\begin{equation} \label{E:astanherd}
 f_P =
 \begin{cases}
 f + \frac{b}{a}g + \frac{c}{a}h,     &\text{if $a \ne 0$;}  \\
 g + \frac{c}{b}h,                    &\text{if $a=0$ and $b \ne 0$;} \\
 h,                                   &\text{if $a = b = 0$.}
 \end{cases}
   \text{  , where } P = \langle a,b,c \rangle .
\end{equation}
The herd selection function here is the one which selects the
representatives of $\langle a,b,c \rangle$ so that the leftmost
non-zero coordinate is 1, i.e. the forms $(1,x,y), (0,1,m) \text{
or} (0,0,1)$. The \emph{normalized herd} is the
$\rho\,\text{-}\mathcal{H}(f,g,h)$ such that
\begin{equation} \label{E:normherd}
 f_P(t) = \frac{af(t) + bg(t) + ch(t)}{af(1) + bg(1) + ch(1)} \text{ where } P = \langle a,b,c \rangle,\forall t \in GF(q) .
\end{equation}
The normalized herd has the property that $f_P(1) = 1, \forall P
\in \mathbf{P}_{\mathcal{HC}}$. The herd selection function for
the normalized herd, unlike the previous examples, involves the
functions $f,g \text{ and } h$ and is given by:
\begin{multline*}
\rho(\langle a,b,c,\rangle) =\\
 \left(
\frac{a}{af(1)+bg(1)+ch(1)},\frac{b}{af(1)+bg(1)+ch(1)},\frac{c}{af(1)+bg(1)+ch(1)}
\right).
\end{multline*}

    We will now examine two special classes of herd spaces which
play a fundamental role in the theory of flocks. It is the
inclusion of these classes which prompts the definition of herd
space that we have given.

    Consider $\mathcal{V} = \mathcal{V}(f, g, h)$. If
$rank(\mathcal{V}) = 1$ (where rank = vector space dimension),
then the functions $f,g \text{ and } h$ are all scalar multiples
of the same function. Suppose then that $f = \alpha F, g = \beta
F, \text{ and } h = \gamma F$ for some function $F \in
\mathcal{Z}$. In this case, (\ref{E:stmap}) becomes $\phi(a,b,c) =
(\alpha a + \beta b + \gamma c)F$ and the kernel of $\hat{\phi}$
is the line $\ell$ with equation $\alpha x + \beta y + \gamma z =
0$ in $\mathbf{P}(\mathcal{U})$. The herd space $\Gamma(\alpha F,
\beta F, \gamma F)$ consists of only two types of ordered pairs,
namely, $(P,\langle 0 \rangle)$ when $P$ is on $\ell$ (here
$\langle 0 \rangle$ denotes the constant function), and
$(P,\langle F \rangle)$ otherwise. The herd cover
$\mathcal{HC}(\alpha F, \beta F, \gamma F)$ has a point set
$\mathbf{P}_{\mathcal{HC}}$ which is either empty if $F$ is not a
permutation or the complement of $\ell$ (an affine subplane of
$\mathbf{P}(\mathcal{U})$) if $F$ is a permutation. The functions
of the $\rho\,$-herds corresponding to the non-empty herd cover
are all scalar multiples of $F$. Herd spaces of this type will be
called \emph{linear herd spaces}.
\begin{proposition} \label{P:lhs}
The point set of the herd cover of a linear herd space is empty
if, and only if, the herd space is degenerate.
\end{proposition}
\begin{proof}
Using the notation of the previous paragraph let
$\mathbf{P}_{\mathcal{HC}}$ be the point set of the herd cover of
the linear herd space $\Gamma =\Gamma(\alpha F, \beta F, \gamma
F)$.  Not all of $\alpha, \beta \text{ and } \gamma$ can be 0,
else $rank(\mathcal{V}) = 0$. So, assume w.l.o.g. that $\alpha
\ne 0$. Now, suppose that this herd space is degenerate. Then
there exist $u \ne v \in GF(q)$ so that $\alpha F(u) = \alpha
F(v)$, and $F$ is not a permutation. On the other hand, if $F$ is
not a permutation then there exist $r \ne s \in GF(q)$ so that
$F(r) = F(s)$. But then, $\alpha F(r) = \alpha F(s), \beta F(r) =
\beta F(s)$ and $\gamma F(r) = \gamma F(s)$, and so, $\Gamma$ is
degenerate.
\end{proof}

    The second case we will consider occurs when $rank(\mathcal{V}) = 2$. In
this case the functions $f,g \text{ and } h$ are linearly
dependent over $GF(q)$, but not all are scalar multiples of the
same function. Thus, there exist constants $\alpha, \beta \text{
and } \gamma$, not all zero, so that $\alpha f + \beta g + \gamma
h = 0$. The kernel of $\hat{\phi}$ is the point $Q = \langle
\alpha,\beta,\gamma \rangle$ of $\mathbf{P}(\mathcal{U})$. Now
consider a line $\ell$ of $\mathbf{P}(\mathcal{U})$ which passes
through $Q$. Such a line has an equation of the form $Ax + By +
Cz = 0$ where not all the coefficients are zero and $A\alpha +
B\beta + C\gamma = 0$. Since not all of $\alpha, \beta \text{ and
} \gamma$ are zero, we can assume w.l.o.g. that $\gamma \ne 0$. If
$P = \langle a,b,c \rangle$ is a point of $\ell$ other than $Q$
then we have,
\begin{align*}
\phi(a,b,c) &= af + bg + ch \\
      &= \frac{1}{\gamma}(a\gamma f + b\gamma g + c\gamma h)\\
      &= \frac{1}{\gamma}((a\gamma - c\alpha)f + (b\gamma -c\beta)g).
\end{align*}
From the fact that both $P$ and $Q$ are on the line $\ell$ we can
derive that $(a\gamma - c\alpha)A + (b\gamma - c\beta)B = 0$.
Now, both of $A$ and $B$ can not be zero, otherwise we would have
that $\gamma C = 0$ and our assumption about $\gamma$ would imply
that $C = 0$. If $A = 0$ then we must have $b\gamma - c\beta = 0$
and so, we obtain:
\begin{equation} \label{E:starherd}
\phi(a,b,c) =
\begin{cases}
\gamma^{-1}(a\gamma - c\alpha)f, &\text{if $A = 0$,} \\
\gamma^{-1}(b\gamma - c\beta)(-\frac{B}{A}f + g)
&\text{otherwise.}
\end{cases}
\end{equation}
In either case, for all the points on $\ell \setminus\{Q\}$, the
associated functions are non-zero multiples of the same non-zero
function. The herd space in this case is described by; $(Q,\langle
0 \rangle)$, and, for each line of $\mathbf{P}(\mathcal{U})$
through $Q$, the points of the line other than $Q$ are associated
with the same point of $\mathbf{P}(\mathcal{V})$. If a point $P$
of $\mathbf{P}(\mathcal{U})$ is in $\mathbf{P}_{\mathcal{HC}}$,
then all the points of the line $PQ$ other than $Q$ are in
$\mathbf{P}_{\mathcal{HC}}$. The herd spaces of this type are
called \emph{proper star herd spaces}. We define a \emph{star
herd space} to be either a proper star herd space or a linear
herd space. Considering that a herd space is a function, we can
rephrase this definition: \emph{a herd space is a star herd space
if, and only if, it is not a bijection.}
\begin{proposition} \label{P:pshs}
The number of lines in the herd cover of a proper star herd space
through the kernel of the herd space is $\mid \mathcal{S} \cap
\mathbf{P}(\mathcal{V}) \mid$.
\end{proposition}
\begin{proof}
Let $\Gamma(f,g,h)$ be a proper star herd space. We shall use the
notation of the above paragraph. Let $m$ be any fixed line of
$\mathbf{P}(\mathcal{U})$ which does not pass through the kernel
$Q$. The restriction of $\hat \phi$ to $m$ has a trivial kernel
and so, is a bijection between $m$ and $\mathbf{P}(\mathcal{V})$.
If $R$ is a point of $m$ such that $\hat\phi\mid_m(R) \in
\mathcal(S)$ then $R \in \mathbf{P}_{\mathcal{HC}}$ and all the
points other than $Q$ of the line $QR$ are in
$\mathbf{P}_{\mathcal{HC}}$. Any line of
$\mathbf{P}(\mathcal{U})$ through $Q$ which is in the herd cover
must intersect $m$ in a point of $\mathbf{P}_{\mathcal{HC}}$,
proving the assertion.
\end{proof}

\section{Flocks and Herd Spaces}

To each flock $\mathcal{F}(f,g,h)$ of $PG(3,q)$ we can naturally
associate the (non-degenerate) herd space $\Gamma(f,g,h)$. By
identifying the points of $\mathbf{P}(\mathcal{U})$ of the herd
space with the points of the plane $x_3=0$ (in the flock) of
$PG(3,q)$ we can attach a geometric significance to the concepts
introduced in the previous section.

Let $\mathcal{F} =\mathcal{F}(f,g,h)$ be a flock of $PG(3,q)$.
Naturally identify the points $\langle a,b,c \rangle \text{ of
}\mathbf{P}(\mathcal{U})$ with the points $\langle a,b,c,0 \rangle
\text{ of the plane }x_3 = 0$. Let $\rho$ be a herd selection
function on $\mathbf{P}(\mathcal{U})$ extended to $x_3 = 0$. For
each point $P$ of $x_3=0$, $\rho$ thus fixes a coordinate
representative for $P$. With $V = (0,0,0,1)$, we fix a coordinate
representative of each point other than $V$ on the line $VP$ of
the form $(a,b,c,\lambda) \text{ where } \rho(P) = (a,b,c,0)$.
Now, each plane $\pi_t$ of $\mathcal{F}$ intersects the line $VP$
in a point other than $V$. We define a function $F_P \colon GF(q)
\to GF(q)$ by $F_P(t) = \lambda$ if $\pi_t \cap VP =
(a,b,c,\lambda)$. Note that for any $P$ we have $F_P(0) = 0$
since $x_3=0$ is always the plane $\pi_0$ of $\mathcal{F}$. The
defining condition is equivalent to $af(t)+ bg(t) + ch(t)
-\lambda = 0$, i.e., $F_P(t) = \lambda = af(t) + bg(t) + ch(t)$.
Thus, the set of ordered pairs $\{(P,\langle F_P \rangle) \mid P
\in x_3 = 0 \}$ defined by the flock $\mathcal{F}$ is clearly
isomorphic to the herd space $\Gamma(f,g,h)$. Note that while the
definition of $F_P$ uses a particular herd selection function, the
resulting $\langle F_P \rangle$ of the herd space is independent
of that choice.

We now see that $F_P$ is a permutation function if, and only if,
no two planes of $\mathcal{F}$ meet the line $VP$ at the same
point. That is to say, $VP$ is a generator line of a cone with
vertex $V$ for which $\mathcal{F}$ is a flock. The set of points
on all such lines form the critical cone of $\mathcal{F}$. Thus,
$\mathbf{P}_{\mathcal{HC}}$ (under the identification) is the
carrier of the critical cone of $\mathcal{F}$. Finally, it should
be clear that $\{F_P \mid P \in \mathbf{P}_{\mathcal{HC}}\}$ is
the $\rho\,$-herd of the herd cover $\mathcal{HC}(f,g,h)$.

Now, we turn to the question of obtaining a flock from a
non-degenerate herd space (degenerate herd spaces do not give
rise to $q$ \emph{distinct} planes). We first note that a herd
space does not determine a unique flock, since for any $k \in
GF(q)^*$ the flocks $\mathcal{F}(f,g,h)$ and
$\mathcal{F}(kf,kg,kh)$ (which are distinct if $k \ne 1$, but
projectively equivalent) give rise to the same herd space.
However, this is the only variation for flocks which give rise to
the same herd space.

\begin{theorem} \label{Th:eqfl}
Two flocks, $\mathcal{F}(f,g,h)$ and $\mathcal{F}(f',g',h')$, give
rise to the same herd space if, and only if, there exists a
non-zero constant $k$ such that $f' = kf, g' = kg$ and $h' = kh$.
\end{theorem}

\begin{proof}
Two flocks, $\mathcal{F}(f,g,h)$ and $\mathcal{F}(f',g',h')$, give
rise to the same herd space if, and only if,
$\hat\phi_{f',g',h'}(P) = \hat\phi_{f,g,h}(P)$ for each point $P
\in \mathbf{P}(\mathcal{U})$. That is, for each point $P = \langle
a,b,c \rangle$ we have $\langle af + bg + ch\rangle = \langle af'
+ bg' + ch'\rangle$. For the points $Q = \langle 1,0,0 \rangle$,
$R = \langle 0,1,0 \rangle$ and $S = \langle 0,0,1 \rangle$ we
have $\langle f \rangle = \langle f' \rangle$, $\langle g \rangle
= \langle g' \rangle$ and $\langle h \rangle = \langle h'
\rangle$ respectively. Thus, there exist constants, $k_Q, k_R$
and $k_S$ so that $f' = k_Q f, g' = k_R g$ and $h' = k_S h$.

Consider $\mathcal{V} = \mathcal{V}(f, g, h)$. If the rank of
$\mathcal{V} = 3$ then at the point $T = \langle 1,1,1 \rangle$ we
have $\langle f + g + h \rangle = \langle f' + g' + h' \rangle$
and so, there is a non-zero constant $k_T$ so that $f'+g'+h' =
k_T(f+g+h)$. Thus, we obtain $(k_Q - k_T)f + (k_R - k_T)g + (k_S
- k_T)h = 0$. Since $f,g$ and $h$ are linearly independent, we
can conclude that $k_Q = k_R = k_S = k_T$. If the rank of
$\mathcal{V} = 2$ then there is a unique point $P = \langle a,b,c
\rangle$ at which we have $af' + bg' + ch' = 0$. This implies
that $ak_Q f + bk_R g + ck_S h = 0$ and so, the kernel of
$\phi_{f,g,h}$ is $(ak_Q,bk_R,ck_S)$. Since the kernel of
$\phi_{f,g,h}$ must equal the kernel of $\phi_{f',g',h'}$ we have
that $\langle a,b,c \rangle = \langle ak_Q, bk_R, ck_S \rangle$
and so we again can conclude that $k_Q = k_R = k_S$. Finally, if
the rank of $\mathcal{V} = 1$ then the functions $f,g$ and $h$
are scalar multiples of each other. We can assume w.l.o.g. that
there exist scalars $\alpha$ and $\beta$ so that $g = \alpha f$
and $h = \beta f$. The kernel of $\phi_{f,g,h}$ is the line with
equation $x + \alpha y + \beta z = 0$ in
$\mathbf{P}(\mathcal{U})$. The kernel of $\phi_{f',g',h'}$ is the
line with equation $k_Q x + \alpha k_R y + \beta k_S z = 0$.
Since these lines must be the same, we have $k_Q = k_R = k_S$.
\end{proof}

A non-degenerate herd space together with a herd selection
function $\rho$ and the values of $\phi(\rho(P))$ for three
non-collinear points $P$ do determine a unique flock. Embed the
projective plane $\mathbf{P}(\mathcal{U})$ of the herd space
$\Gamma(f,g,h)$ in $PG(3,q)$ as the plane $x_3=0$ by $\langle
a,b,c \rangle \mapsto \langle a,b,c,0 \rangle$. The functions
$f,g \text{ and } h$ can be obtained by using $\rho$ as follows:
Let $P_1, P_2$ and $P_3$ be the three given non-collinear points.
If $\rho(P_1) = (a_1,a_2,a_3)$, $\rho(P_2) = (b_1,b_2,b_3)$,
$\rho(P_3) = (c_1,c_2,c_3)$ and $\phi(\rho(P_i)) = f_i, i = 1,2,3
$ we have:
 \begin{equation*}
 \left(
    \begin{matrix}
    f \\
    g\\
    h
    \end{matrix}
    \right) = \left(
    \begin{matrix}
    a_1 & a_2 & a_3  \\
    b_1 & b_2 & b_3  \\
    c_1 & c_2 & c_3
    \end{matrix}
    \right)^{-1} \left(
    \begin{matrix}
    f_1 \\
    f_2\\
    f_3
    \end{matrix}
    \right).
  \end{equation*}
\noindent The matrix inverse exists since the points are
non-collinear. Now we can form the flock $\mathcal{F} =
\mathcal{F}(f,g,h)$. Note that we obtain $q$ planes since
$\Gamma(f,g,h)$ is non-degenerate.

In the construction of a herd space from a flock, a herd
selection function was used to define the functions $F_P$. If
this herd selection function is used with the constructed herd
space to obtain a flock, then it should be clear that the original
flock is recaptured if the functions $F_P$ are used as the $f_i$
selections.

The above discussion can be summarized as,

\begin{theorem}
  Any flock $\mathcal{F}(f,g,h)$ gives rise to a unique herd space
  $\Gamma(f,g,h)$. On the other hand, any non-degenerate herd space
    $\Gamma(f,g,h)$ gives rise to several flocks related as in
    Theorem \ref{Th:eqfl}. A unique flock can be constructed from
    a herd space if a herd selection function $\rho$ is given and the value
    of $\phi(\rho(P))$ is known for three non-collinear points P.
    For each flock that can be constructed from a herd space, the critical cone of
    the flock has $\mathbf{P}_{\mathcal{HC}}$ as its carrier.
\end{theorem}

\begin{corollary}[The General Herd Theorem]
If the herd cover of a non-degenerate herd space $\Gamma(f,g,h)$
has a point set which contains three non-collinear points, then
any $\rho\,$-herd of this herd cover gives rise to a unique flock
$\mathcal{F}$ whose critical cone has $\mathbf{P}_{\mathcal{HC}}$
as its carrier. Conversely, for any flock $\mathcal{F}(f,g,h)$ in
$PG(3,q)$ and any herd selection function $\rho$, there is a
$\rho\,$-herd for which the the corresponding
$\mathbf{P}_{\mathcal{HC}}$ is the carrier of the critical cone
of the flock.
\end{corollary}

\begin{proof}
Since the point set of the herd cover contains three
non-collinear points, the $\rho\,$-herd will contain at least
three ordered pairs of the form $(\rho(P),\phi(\rho(P))$
corresponding to these points. This is the data which is needed to
construct a unique flock from a non-degenerate herd space. The
converse is obvious.
\end{proof}

In light of this theorem, we will refer to the (non-degenerate)
herd space $\Gamma(f,g,h)$ as being the \emph{herd space of}
$\mathcal{F} = \mathcal{F}(f,g,h)$, since the choice of $\rho$ is
immaterial. We will refer to the $\rho\,$-herd of
$\mathcal{HC}(f,g,h)$ as being the \emph{ herd of $\mathcal{F}$},
only when the herd selection function $\rho$ is clearly
understood. It will also be convenient to abuse notation and
refer to a representative of a point of $\mathbf{P}(\mathcal{V})$
as being a \emph{function of the herd space}. \\

\section{Equivalence of Herd Spaces}

Two herd spaces $\Gamma(f,g,h) \text{ and } \Gamma(f',g',h')$ are
defined to be \emph{equivalent} if there exist a collineation
$\psi \colon \mathbf{P}(\mathcal{U}) \to \mathbf{P}(\mathcal{U})$
and a collineation $\tau \colon \mathbf{P}(\mathcal{Z}) \to
\mathbf{P}(\mathcal{Z})$ with $\tau( \langle f,g,h \rangle ) =
\langle f',g',h' \rangle$  such that the following diagram
commutes,
\begin{equation} \label{D:comm}
 \begin{CD}
\mathbf{P}(\mathcal{U})  @>\hat\phi_{f,g,h}>>   \langle f,g,h \rangle\\
@V\psi VV                                          @V\tau VV\\
\mathbf{P}(\mathcal{U})  @>\hat\phi_{f',g',h'}>>  \langle
f',g',h'\rangle
 \end{CD}
\end{equation}

\noindent and $\psi(\mathcal{P}_{\mathcal{HC}(f,g,h)}) =
\mathcal{P}_{\mathcal{HC}(f',g',h')}$. If the two herd spaces are
equivalent, and $\mathcal{P}_{\mathcal{HC}(f,g,h)} =
\mathcal{P}_{\mathcal{HC}(f',g',h')}$ then they are said to be
\emph{strongly equivalent}. If, in addition to being strongly
equivalent, $\tau$ preserves $\mathcal{S}$, then $\Gamma(f,g,h)$
and $\Gamma(f',g',h')$ are said to be \emph{herd equivalent}.
Clearly, herd equivalence implies strong equivalence which in
turn implies equivalence. Two flocks are said to be \emph{(herd,
strongly) equivalent} if their associated herd spaces are (herd,
strongly) equivalent.

\begin{lemma} \label{L:comm}
    If in the definition of equivalent herd spaces, diagram
    (\ref{D:comm}) commutes, then $\psi(\mathcal{P}_{\mathcal{HC}(f,g,h)}) =
\mathcal{P}_{\mathcal{HC}(f',g',h')}$ if and only if $\tau(
\mathcal{S} \cap \langle f,g,h \rangle ) = \mathcal{S} \cap
\langle f',g',h' \rangle$. In particular, if $\tau = id$ then the
condition $\psi(\mathcal{P}_{\mathcal{HC}(f,g,h)}) =
\mathcal{P}_{\mathcal{HC}(f',g',h')}$ is superfluous.
\end{lemma}
\begin{proof}
Let $T = \mathcal{P}_{\mathcal{HC}(f,g,h)}$ and $T' =
\mathcal{P}_{\mathcal{HC}(f',g',h')}$. By definition,
$\hat\phi_{f,g,h}(T) = \mathcal{S} \cap \langle f,g,h \rangle$
and $\hat\phi_{f',g',h'}(T') = \mathcal{S} \cap \langle f',g',h'
\rangle$. Commutativity of diagram (\ref{D:comm}) gives
$\tau(\hat\phi_{f,g,h}(T)) = \hat\phi_{f',g',h'}(\psi(T))$. The
RHS = $\mathcal{S} \cap \langle f',g',h' \rangle$ if and only if
$\psi(T) = T'$.
\end{proof}

While it is  true that star flocks are not generally equivalent,
we do have:

\begin{proposition} \label{P:linear3}
    All linear flocks are equivalent.
\end{proposition}
\begin{proof}
Let $\mathcal{F}(f,g,h)$ and $\mathcal{F}(f',g',h')$ be linear
flocks. $\langle f,g,h \rangle$ and $\langle f',g',h' \rangle$
are points of $\mathbf{P}(\mathcal{Z})$ and since the
automorphism group of $\mathbf{P}(\mathcal{Z})$ acts transitively
on its points, we can find a $\tau$ in this group with $\tau(
\langle f,g,h \rangle ) = \langle f',g',h' \rangle$. The kernels
of $\Gamma(f,g,h)$ and $\Gamma(f',g',h')$ are lines $\ell$ and
$\ell'$ of $\mathbf{P}(\mathcal{U})$ respectively. Since the
automorphism group of $\mathbf{P}(\mathcal{U})$ acts transitively
on its lines, we can find an automorphism $\psi$ of
$\mathbf{P}(\mathcal{U})$ with $\psi(\ell) = \ell'$. Since the
point sets of the herd covers of these two linear herd spaces are
the complements, in $\mathbf{P}(\mathcal{U})$, of these kernels
we have $\psi(\mathcal{P}_{\mathcal{HC}(f,g,h)}) =
\mathcal{P}_{\mathcal{HC}(f',g',h')}$. Commutativity of diagram
(\ref{D:comm}) with this choice of $\psi$ and $\tau$ is obvious.
\end{proof}

\begin{proposition} \label{P:coef}
    For any non-zero constants $A, B \text{ and } C$ the flocks
    $\mathcal{F}(f,g,h)$ and $\mathcal{F}(Af,Bg,Ch)$ are
    equivalent. If $A = B = C$ then the corresponding herd spaces
    are herd equivalent (in fact, equal).
\end{proposition}
\begin{proof}
Since $\langle f,g,h \rangle = \langle Af,Bg,Ch \rangle$ we may
take $\tau = id$. Let $\psi$ be the homography defined by
$\psi(\langle x_0,x_1,x_2 \rangle) = \langle
\frac{x_0}{A},\frac{x_1}{B},\frac{x_2}{C} \rangle$. Consider the
point $P = \langle a,b,c \rangle$ in $\mathbf{P}(\mathcal{U})$.
$\tau(\hat\phi_{f,g,h}(P)) = \langle af + bg + ch \rangle$ and
$\hat\phi_{Af,Bg,Ch}(\psi(P)) = \langle\frac{a}{A}(Af) +
\frac{b}{B}(Bg) +\frac{c}{C}(Ch)\rangle = \langle af + bg + ch
\rangle$, so diagram (\ref{D:comm}) commutes. The result now
follows from Lemma \ref{L:comm}. If $A = B = C$ then $\psi$ is
just the identity map of $\mathbf{P}(\mathcal{U})$.
\end{proof}

The last result is a special case of the following:

\begin{proposition} \label{P:perm}
    If $\mathcal{F} =\mathcal{F}(f,g,h)$ is a non-star flock and $\{f',g',h'\}$
    is any basis of $\langle f,g,h \rangle$ then
    $\mathcal{F}(f',g',h')$ is equivalent to $\mathcal{F}$.
\end{proposition}
\begin{proof}
Since $\langle f,g,h \rangle = \langle f',g',h' \rangle$ we may
take $\tau = id$. Let $\psi$ be the homography of
$\mathbf{P}(\mathcal{U})$ given by the change of basis matrix for
the ordered basis $\{f',g',h'\}$ to the ordered basis $\{f,g,h\}$
acting on points. A straight-forward calculation shows that
diagram (\ref{D:comm}) commutes, and the result follows from
Lemma \ref{L:comm}.
\end{proof}

An analogous result for star flocks will be given in the next
section.

\begin{proposition} \label{P:steq}
    Let $\mathcal{F}$ and $\mathcal{F'}$ be two flocks with the same
    critical cone $\mathcal{K}$ in $PG(3,q)$ whose carrier contains
    at least 3 non-collinear points. If there is a collineation of
    $PG(3,q)$ stabilizing $\mathcal{K}$ and mapping $\mathcal{F}$
    to $\mathcal{F'}$ then $\mathcal{F}$ and
    $\mathcal{F'}$ are strongly equivalent.
\end{proposition}
\begin{proof}
We will provide a proof only for the case that $\mathcal{F}$ and
$\mathcal{F'}$ are non-star flocks and leave the star flock case
to the reader.

 Let $\mathcal{F} =\mathcal{F}(f,g,h)$ and $\mathcal{F}'
 =\mathcal{F}(f',g',h')$ and suppose that $\sigma$ is a
collineation of $PG(3,q)$ stabilizing $\mathcal{K}$ and mapping
$\mathcal{F}$ to $\mathcal{F'}$. $\sigma$ induces a map that maps
$f \mapsto f',\; g \mapsto g'$ and $h \mapsto h'$. Since
$\mathcal{F}$ is a non-star flock, $f, g$ and $h$ are linearly
independent vectors of $\mathcal{Z}$ and this set can be extended
to an ordered basis $\mathcal{B} = \{\mu_0 = f, \mu_1 = g, \mu_2
= h, \mu_3, \ldots, \mu_{q-2}\}$ of $\mathcal{Z}$. Similarly, we
can construct an ordered basis $\mathcal{B}' = \{\mu_0' = f',
\mu_1' = g', \mu_2' = h', \mu_3', \ldots, \mu_{q-2}'\}$. The
linear extension of the map given by $\mu_i \mapsto \mu_i',\; 0
\le i \le q-2$, is a vector space automorphism. If we let $\tau$
be the collineation of $\mathbf{P}(\mathcal{Z})$ induced by this
automorphism, it is clear that $\tau(\langle f,g,h \rangle) =
\langle f',g',h' \rangle$. Since $\tau(\langle af+bg+ch \rangle) =
\langle af'+bg'+ch' \rangle$ for all $a,b,c \in GF(q)$, if we
take $\psi = id$ then diagram (\ref{D:comm}) commutes. As
$\mathcal{F}$ and $\mathcal{F'}$ have the same critical cone, the
point sets of the herd covers of their herd spaces are identical
and so, preserved by $\psi$. Thus, $\mathcal{F}$ and
$\mathcal{F'}$ are strongly equivalent.
\end{proof}

Let $\mathcal{F} = \{\pi_t \mid t \in GF(q)\}$ be a flock of
planes. A \emph{reparameterization} of $\mathcal{F}$ is a
reassignment of the elements of $GF(q)$ to the planes of
$\mathcal{F}$, with $0$ still assigned to the plane $x_3=0$. If
after a reparameterization we have $\mathcal{F} = \{\pi_s \mid s
\in GF(q)\}$ then there exists a permutation $r$ of $GF(q)$ so
that $\pi_t = \pi_{r(s) - r(0)}$. In terms of coordinate
functions, if $\mathcal{F} = \mathcal{F}(f,g,h)$ then after this
reparameterization we would have $\mathcal{F} =
\mathcal{F}(f',g',h')$ where $f'(s) = f(r(s)-r(0))$, $g'(s) =
g(r(s)-r(0))$ and $h'(s) = h(r(s)-r(0))$. By changing the
parameterization we obtain a different set of coordinate
functions, but we haven't changed the planes, so the
corresponding herd spaces should be equivalent in a very strong
sense. This is the content of the next proposition.

\begin{proposition} \label{P:reindex}
    The herd space of a flock and the herd space of any
    reparameterization of that flock are herd equivalent.
\end{proposition}
\begin{proof}
Let $\mathcal{F}(f',g',h')$ be a reparameterization of
$\mathcal{F}(f,g,h)$ by the permutation $r$ of $GF(q)$. That is,
$f'(s) = f(r(s)-r(0))$, $g'(s) = g(r(s)-r(0))$ and $h'(s) =
h(r(s)-r(0))$. Now, $p(s) = r(s) - r(0)$ for $s \in GF(q)$ is also
a permutation of $GF(q)$. Define $\tau \colon
\mathbf{P}(\mathcal{Z}) \to \mathbf{P}(\mathcal{Z})$ by
$\tau(\langle F \rangle) = \langle F\circ p \rangle$. This is
easily seen to be well defined and obviously, $F\circ p(0) = 0$.
$\tau$ is a collineation of $\mathbf{P}(\mathcal{Z})$ and since
the composition of permutations is a permutation,
$\tau(\mathcal{S}) = \mathcal{S}$. Clearly, $\tau(\langle f,g,h
\rangle) = \langle f',g',h' \rangle$. Let $\psi = id$. Since
$\tau(\langle af(t) + bg(t) +ch(t)\rangle) = \langle af(p(s)) +
bg(p(s)) + ch(p(s))\rangle = \hat\phi_{f',g',h'}(\langle a,b,c
\rangle)$, diagram (\ref{D:comm}) commutes.
\end{proof}

In order to compare our results with the extant literature
concerning flocks of quadratic (and other) cones, the following
definition will be useful. Let $\mathcal{C}$ be a subset of
$\mathcal{P}_{\mathcal{HC}(f,g,h)}$ of the herd space
$\Gamma(f,g,h)$. The herd spaces $\Gamma(f,g,h) \text{ and }
\Gamma(f',g',h')$ are defined to be
$\mathcal{C}-$\emph{equivalent} if in the definition of
equivalent herd spaces the condition
$\psi(\mathcal{P}_{\mathcal{HC}(f,g,h)}) =
\mathcal{P}_{\mathcal{HC}(f',g',h')}$ is replaced by
$\psi(\mathcal{C}) \subseteq
\mathcal{P}_{\mathcal{HC}(f',g',h')}$. These herd spaces are said
to be \emph{strongly $\mathcal{C}-$equivalent} if they are
$\mathcal{C}-$equivalent and $\psi(\mathcal{C}) = \mathcal{C}$. We
refer to two flocks as being \emph{(strongly)
$\mathcal{C}-$equivalent} if their herd spaces are (strongly)
$\mathcal{C}-$equivalent. Note that equivalent herd spaces are
$\mathcal{C}-$equivalent for any appropriate set $\mathcal{C}$,
but the converse need not be true. Consider the following example
of $\mathcal{C}-$equivalence.

\begin{example}
The flocks $\mathcal{F}(t, 5t^3,5t^5)$ and
$\mathcal{F}(t+7t^3+3t^5, 5t^3+14t^5, 5t^5)$ of $PG(3,17)$ are
equivalent, but not strongly equivalent flocks. The equivalence
is given by the collineations $\psi(x_0,x_1,x_2) = (x_0, 2x_0+x_1,
4x_0+4x_1+x_2)$ and $\tau = id$. $\mathcal{P}_{\mathcal{HC}(t,
5t^3,5t^5)}$ consists of the conic $\mathcal{C}\colon x_0x_2 =
x_1^2$ together with the point $(0,1,0)$, while
$\mathcal{P}_{\mathcal{HC}(t+7t^3+3t^5, 5t^3+14t^5, 5t^5)}$
consists of the conic $\mathcal{C}\colon x_0x_2 = x_1^2$ together
with the point $(0,13,1)$. Since these point sets are not equal,
the flocks are not strongly equivalent. However, since $\psi$
stabilizes the conic $\mathcal{C}$, these two flocks are strongly
$\mathcal{C}$-equivalent.
\end{example}

\begin{proposition}
    Let $\mathcal{C}$ be a set of points in the plane $x_3 = 0$ of
    $PG(3,q)$ whose automorphism group acts transitively on the
    lines of $x_3 = 0$ which do not meet $\mathcal{C}$. Then, all
    linear flocks of the cone of $PG(3,q)$ with carrier
    $\mathcal{C}$ are strongly $\mathcal{C}-$equivalent.
\end{proposition}
\begin{proof}
Let $\mathcal{F}(f,g,h)$ and $\mathcal{F}(f',g',h')$ be linear
flocks of the cone with carrier $\mathcal{C}$. $\langle f,g,h
\rangle$ and $\langle f',g',h' \rangle$ are points of
$\mathbf{P}(\mathcal{Z})$ and since the automorphism group of
$\mathbf{P}(\mathcal{Z})$ acts transitively on its points, we can
find a $\tau$ in this group with $\tau( \langle f,g,h \rangle ) =
\langle f',g',h' \rangle$. We identify the plane $x_3 = 0$ with
$\mathbf{P}(\mathcal{U})$ in the usual manner. The kernels of
$\Gamma(f,g,h)$ and $\Gamma(f',g',h')$ are lines $\ell$ and
$\ell'$ respectively, of $\mathbf{P}(\mathcal{U})$ which do not
intersect $\mathcal{C}$. By the hypothesis we can find a $\psi$
in the automorphism group of $\mathcal{C}$ so that $\psi(\ell) =
\ell'$. Commutativity of diagram (\ref{D:comm}) is obvious. Since
$\psi(\mathcal{C}) = \mathcal{C} \subseteq
\mathcal{P}_{\mathcal{HC}(f',g',h')}$, we see that
$\mathcal{F}(f,g,h)$ and $\mathcal{F}(f',g',h')$ are strongly
$\mathcal{C}-$equivalent.
\end{proof}

In most of the applications, $\mathcal{C}$ will be a conic or
other oval of $x_3 = 0$. When $\mathcal{C}$ is a conic, strong
$\mathcal{C}-$equivalence coincides with the concept of flock
equivalence found in the literature on the quadratic cone case.
The proposition above is thus a generalization of the well known
result that all linear flocks of quadratic cones are
``equivalent'', while Proposition \ref{P:linear3} is not. From a geometrical point of view, all linear flocks \textit{should} be equivalent since they are structurally the same, but this will not be the case for non-quadratic cones, in general, if the notion of equivalence requires that the cone be stabilized. This consideration has led us to the more general notion of equivalence that we have adopted.

\section{Star Flocks}

We can easily characterize star flocks in terms of their herd
spaces. A more detailed examination of star flocks can be found
in \cite{WEC2}.

\begin{proposition} \label{P:star1}
   A flock $\mathcal{F}$ is a star flock if, and only if, its herd
   space contains the constant function. A herd space is a star herd space if and only if
   its associated flock is a star flock.
\end{proposition}
\begin{proof}
Let $\mathcal{F}$ be a star flock with the point $Q$ common to
all the planes of $\mathcal{F}$. Since $x_3=0$ is in the flock,
$Q$ is in this plane. Clearly, $f_Q(t) = 0, \forall t \in GF(q)$,
so the herd space of $\mathcal{F}$ contains a constant function
(i.e., $(Q,\langle 0 \rangle)$). On the other hand, if the herd
space of $\mathcal{F}$ contains a constant function, it must be
the zero function. The point associated with the zero function
lies in all planes of the flock, and so, the flock is a star
flock. Since the kernel of $\hat\phi$ is not trivial, the herd
space in this case is a star herd space.
\end{proof}

\begin{proposition} \label{P:linear1}
   A flock $\mathcal{F}$ is a linear flock if and only if its
   herd space contains at least two constant functions. In this
   case, the herd space is a linear herd space.
\end{proposition}
\begin{proof}
Suppose that the herd space of flock $\mathcal{F}$ contains
$(P,\langle 0 \rangle)$ and $(Q,\langle 0 \rangle)$ with $P \ne
Q$. Then by Proposition \ref{P:star1} both $P$ and $Q$, and hence
the entire line $PQ$, lie in all planes of $\mathcal{F}$. The
converse is clear. The kernel of $\hat\phi$ has dimension 1, and
the herd space is a linear herd space.
\end{proof}

\begin{proposition} \label{P:starequiv}
    Any star flock is equivalent to one of the form
    $\mathcal{F}(f,\langle 0 \rangle, h)$. Any linear flock is
    equivalent to one of the form $\mathcal{F}(f,\langle 0 \rangle, \langle 0
    \rangle)$.
\end{proposition}
\begin{proof}
Consider the flock $\mathcal{F} =\mathcal{F}(f,\langle 0 \rangle,
\langle 0 \rangle)$ where $f$ is a non-constant function. By
Proposition \ref{P:linear1} $\mathcal{F}$ is a linear flock. If
$\mathcal{F}'$ is any linear flock, then $\mathcal{F}'$ is
equivalent to $\mathcal{F}$ by Proposition \ref{P:linear3}.

 Now,
suppose that $\mathcal{F} =\mathcal{F}(f,g,h)$ is a proper star
flock. $\langle f,g,h \rangle$ is a line of
$\mathbf{P}(\mathcal{Z})$ and there exist constants $a,b$ and $c$,
not all zero, so that $af + bg + ch = 0$. The kernel of
$\Gamma(f,g,h)$ is the point $Q = \langle a,b,c \rangle$ in
$\mathbf{P}(\mathcal{U})$. Consider the points $P = \langle 1,0,0
\rangle$ and $R = \langle 0,0,1 \rangle$ of
$\mathbf{P}(\mathcal{U})$. If $Q$ is not on the line $PR$, i.e.,
$b \ne 0$, then the homography $\psi_1 \in PGL(3,q)$ given by the
matrix,
  \begin{equation}
    \left(
    \begin{matrix}
    1 & 0 & 0  \\
    -\frac{a}{b} & 1 & -\frac{c}{b}  \\
    0 & 0 & 1
    \end{matrix}
    \right),
  \end{equation}
\noindent acting on points, fixes $P$ and $R$ and $\psi_1(Q) =
\langle 0,1,0 \rangle$. Taking $\tau = id$ and $\psi = \psi_1$
shows that the flock $\mathcal{F}$ is equivalent to
$\mathcal{F}(f,\langle 0 \rangle, h)$. If $Q$ is on the line $PR$
but not equal to  $P$ (thus, $b = 0, c \ne 0$), then the
homography $\psi_2 \in PGL(3,q)$ given by the matrix,
  \begin{equation}
    \left(
    \begin{matrix}
    1 & 0 & 0  \\
    0 & 0 & 1 \\
    -\frac{a}{c} & 1 & 0
    \end{matrix}
    \right),
  \end{equation}
\noindent acting on points, fixes $P$, $\psi_2(Q) = \langle 0,1,0
\rangle$ and $\psi_2(\langle 0,1,0 \rangle) = R$. Taking $\tau =
id$ and $\psi = \psi_2$ shows that the flock $\mathcal{F}$ is
equivalent to $\mathcal{F}(f,\langle 0 \rangle, g)$. Finally, if
$Q = P$, then the homography $\psi_3 \in PGL(3,q)$ given by the
matrix,
  \begin{equation}
    \left(
    \begin{matrix}
    0 & 1 & 0  \\
    1 & 0 & 0 \\
    0 & 0 & 1
    \end{matrix}
    \right),
  \end{equation}
\noindent acting on points, fixes $R$, $\psi_3(Q) = \langle 0,1,0
\rangle$ and $\psi_3(\langle 0,1,0 \rangle) = P$. Taking $\tau =
id$ and $\psi = \psi_3$ shows that the flock $\mathcal{F}$ is
equivalent to $\mathcal{F}(g,\langle 0 \rangle, h)$.
\end{proof}

We can also characterize star flocks in terms of other functions
in their herd spaces.

\begin{proposition} \label{P:star2}
    A flock $\mathcal{F}$ is a star flock if and only if the
    function classes associated to two distinct points in its herd
    space are equal, i.e., the associated functions are scalar
    multiples.
\end{proposition}
\begin{proof}
The herd space of the flock $\mathcal{F}$ contains $(P,\langle f
\rangle)$ and $(Q, \langle f \rangle)$ with $P \ne Q$ if and only
if the herd space is not a bijection. The herd space is a star
herd space and so, its kernel is not empty. The statement now
follows from Proposition \ref{P:star1}.
\end{proof}

\begin{proposition} \label{P:linear2}
    A flock $\mathcal{F}$ is a linear flock if and only if there
    exist three non-collinear points in its herd space whose associated
    function classes are non-constant and equal.
\end{proposition}
\begin{proof}
If $\mathcal{F}$ is a linear flock then all points of the herd
cover of its herd space are associated to the same permutation.
Choose any three non-collinear points of the affine plane which
is the point set of the herd cover to satisfy the condition. On
the other hand, suppose that $(P,\langle f \rangle), (S, \langle
f \rangle)$ and $(R, \langle f \rangle)$ are in the herd space of
$\mathcal{F}$ with $P, S$ and $R$ distinct non-collinear points of
$\mathbf{P}(\mathcal{U})$ and $f$ a non-constant function. By
Proposition \ref{P:star2}, $\mathcal{F}$ is a star flock. The
restriction of $\hat\phi$ to the line $PS$ is not a bijection, so
this restriction has a nontrivial kernel. Thus, there is a point
$Q_1$ on $PS$ which is in the kernel of $\hat\phi$. Similarly,
there is a point $Q_2$ on the line $PR$ in the kernel of
$\hat\phi$. If $Q_1 = Q_2$ then they would be the point $P$
contradicting the fact that $P$ is not in the kernel since $f$ is
non-constant. Thus, there are at least two constant functions in
the herd space and the result follows from Proposition
\ref{P:linear1}.
\end{proof}

\section{Some Non-Star Flocks}

In this section we will present some examples of the herd spaces
of various non-star flocks. Only a selected few are examined to
illustrate the techniques and ideas concerning herds and herd
spaces. To determine the herd covers of these herd spaces, we
appeal to Dickson \cite{LED:01} for the required information on
permutation polynomials. He has determined all permutation
polynomials of degree $\le 5$ and we rely heavily on this
classification.

    Consider the herd space $\Gamma(t,t^2,t^3)$. The functions of
this herd space are of the forms: $t^3 + at^2 + bt,\, at^2 + t
\text{ and }t^2$ where $a,b \in GF(q)$.  $t^2$ is a permutation
polynomial if, and only if $q = 2^e$. That is to say, $(0,1,0)
\in \mathbf{P}_{\mathcal{HC}} \text{ iff }q=2^e$. $at^2 + t$ is a
permutation polynomial iff $a = 0$, in which case it is a
permutation polynomial for all $q$. Finally, there are two cases
for which $t^3 + at^2 + bt$ is a permutation polynomial. The
first case occurs when $q \equiv -1 \bmod{3}$ and $a = 3c, b =
3c^2 (\forall c \in GF(q))$, and the second when $q = 3^e$ and $a
= 0, b = -n \text{ for } n$ a non-square in $GF(q)$. We summarize
the possibilities for $\mathbf{P}_{\mathcal{HC}}$ of
$\mathcal{HC}(t,t^2,t^3)$ for all $q \ge 5$ in Table
\ref{Ta:one}. For completeness, note that all flocks with $q < 5$
having more than two points in the carriers of their critical
cones are star flocks, as will be shown in the next section.
\begin{table}
 \begin{center}
  \begin{tabular}{| c | c |}
   \hline
   \emph{q} & $\mathbf{P}_{\mathcal{HC}}$ \\ \hline \hline
   $q \equiv 1 \bmod{3},\, q \text{ odd }$ & $(1,0,0)$ \\ \hline
   $q \equiv 1 \bmod{3},\, q = 2^{2k}$ & $(1,0,0), (0,1,0)$ \\ \hline
   $q \equiv -1 \bmod{3},\, q \text{ odd }$ & \emph{conic}, $3x_0x_2 = {x_1}^2$ \\ \hline
   $q \equiv -1 \bmod{3},\, q =2^{2k+1}$ & \emph{hyperconic}, $x_0x_2 = {x_1}^2 \cup (0,1,0)$ \\ \hline
   $q = 3^e$ & $(1,0,0), (0,0,1), \{(-n, 0, 1) \mid n \text{ a non-square }\}$ \\ \hline
  \end{tabular}
  \caption{$\mathbf{P}_{\mathcal{HC}}\text{ of }
\mathcal{HC}(t,t^2,t^3),\, q \ge 5$}\label{Ta:one}
  \end{center}
\end{table}

    We illustrate a few herds of the flock $\mathcal{F}(t,t^2,t^3)$, in the case
that $q \equiv -1 \bmod{3},\, q \text{ odd }$. The standard herd
is $\{f_{\langle 1,0,0,0 \rangle}(t) = t \} \cup \{f_{\langle
3c^2, 3c, 1, 0 \rangle}(t) = 3c^2t + 3ct^2 + t^3 \mid c \in
GF(q)\}$, while the alternate standard herd is given by
$\{f_{\langle 1,0,0,0 \rangle}(t) = t \} \cup \{f_{\langle 3c^2,
3c, 1, 0 \rangle}(t) = t + \frac{1}{c}t^2 + \frac{1}{3c^2}t^3
\mid c \in GF(q)^*\} \cup \{f_{\langle 0,0,1,0 \rangle}(t) = t^3
\}$. The normalized herd in this case is $\{f_{\langle 1,0,0,0
\rangle}(t) = t \} \cup \{f_{\langle 3c^2, 3c, 1, 0 \rangle}(t) =
\frac{3c^2t + 3ct^2 + t^3}{3c^2 + 3c + 1} \mid c \in GF(q)\}$.

Continuing with this example, we note that the homography of
$PG(3,q)$, $q \ne 3^e$ represented by the matrix,
  \begin{equation} \label{E:ftw}
    \left(
    \begin{matrix}
    1 & 0 & 0 & 0 \\
    0 & 3 & 0 & 0 \\
    0 & 0 & 3 & 0 \\
    0 & 0 & 0 & 1
    \end{matrix}
    \right),
  \end{equation}
acting on planes (on the left), will map the flock
$\mathcal{F}(t,t^2,t^3)$ to the flock $\mathcal{F}(t,3t^2,3t^3)$
and the conic $3x_0x_2 = {x_1}^2$ to the conic $x_0x_2 = {x_1}^2$
in the plane $x_3 = 0$ (which is stabilized by this
collineation). Thus, for $q \equiv -1 \bmod{3}$,
$\mathcal{F}(t,t^2,t^3)$ is  equivalent (Proposition \ref{P:coef})
to the Fisher-Thas-Walker (FTW) flock of a quadratic cone as
represented in \cite{JoPa:97}. If $q = 2^{2k+1}$, for $k \ge 1$,
then $t \mapsto t^{\frac{1}{4}}$ is a permutation of $GF(q)$
which fixes $0$. If we reparameterize the planes in the the flock
$\mathcal{F}(t,t^2,t^3)$ using this permutation, we obtain the
herd equivalent flock
$\mathcal{F}(t^{\frac{1}{4}},t^{\frac{1}{2}},t^{\frac{3}{4}})$
(Proposition \ref{P:reindex}). The herd space of this flock is
herd equivalent to the original herd space, and so, has the same
$\mathbf{P}_{\mathcal{HC}}$. The normalized herd of this
 flock is therefore, $\{f_{\langle 1,0,0,0
\rangle}(t) = t^{\frac{1}{4}}, f_{\langle 0,1,0,0 \rangle}(t) =
t^{\frac{1}{2}} \} \cup \{f_{\langle c^2, c, 1, 0 \rangle}(t) =
\frac{c^2t^{\frac{1}{4}} + ct^{\frac{1}{2}} +
t^{\frac{3}{4}}}{c^2 + c + 1} \mid c \in GF(q)\}$. Observe that
all the functions in the normalized herd are o-polynomials. This
herd (without $f_{\langle 0,1,0,0 \rangle}$, and with a different
indexing) is called a \emph{herd of ovals} in \cite{CPPR}.

    We now consider the more complex herd space of the flock $\mathcal{F}(t,t^3,t^5)$.
The functions of this herd space are of the forms: $t^5 + at^3 +
bt,\, at^3 + t \text{ and }t^3$ where $a,b \in GF(q)$. Using
Dickson \cite{LED:01}, we may again list all possible herd covers
of this herd space. For $q \le 5$, this flock is a star flock.
The possibilities for $\mathbf{P}_{\mathcal{HC}}$ of
$\mathcal{HC}(t,t^3,t^5)$ for all $q \ge 7$ are summarized in
Table \ref{Ta:two}.
\begin{table}[htb!]
 \begin{center}
  \begin{tabular}{| c | c |}
   \hline
   \emph{q} & $\mathbf{P}_{\mathcal{HC}}$ \\ \hline \hline
   $q \equiv 1 \pod{15}$ & $(1,0,0)$ \\ \hline
   $q \equiv 2,8 \pod{15}$ & $(0,1,0)$, \emph{conic}  \\ \hline
   $q \equiv 3,12 \pod{15},\, q=3^{2k+1}$ &$(0,1,0)$, \emph{conic}, \emph{collinear set}$_1$ \\ \hline
   $q \equiv 4 \pod{15}$ & $(1,0,0),(0,0,1)$\\ \hline
   $q \equiv 5 \pod{15},\,q=5^{2k+1}$ & $(0,1,0)$, \emph{collinear set}$_2$, \emph{partial conic}   \\ \hline
   $q \equiv 6 \pod{15},\,q=3^{4k}$ & $(0,1,0)$, \emph{collinear set}$_1$\\ \hline
   $q \equiv 7,13 \pod{15}$  & \emph{conic}    \\ \hline
   $q \equiv 9 \pod{15},\,q=3^{4k+2}$ & $(1,0,0),(0,1,0),(0,0,1)$,\emph{collinear set}$_1$\\ \hline
   $q \equiv 10 \pod{15},\,q=5^{2k}$ & \emph{collinear set}$_2$, \emph{partial conic} \\ \hline
   $q \equiv 11 \pod{15}$ & $(1,0,0), (0,1,0)$\\ \hline
   $q \equiv 14 \pod{15}$ & $(1,0,0), (0,1,0), (0,0,1)$\\ \hline
   \hline
   \multicolumn{2}{| c |}{\emph{conic} : $5x_0x_2 = {x_1}^2$} \\
   \hline
   \multicolumn{2}{| c |}{\emph{collinear set}$_1$ : $(-n,1,0)$, n a
   non-square} \\ \hline
   \multicolumn{2}{| c |}{\emph{collinear set}$_2$ : $(-s,0,1), \text{ s } \ne \,4^{th} \text{
   power}$} \\ \hline
   \multicolumn{2}{| c |}{\emph{partial conic} : $4x_0x_2 = x_1^2 \text{ with } x_1 = 2n$, n a
   non-square} \\ \hline
  \end{tabular}
  \caption{$\mathbf{P}_{\mathcal{HC}}\text{ of }
\mathcal{HC}(t,t^3,t^5),\, q \ge 7$}\label{Ta:two}
  \end{center}
\end{table}

As we see from the table, the only cases in which
$\mathbf{P}_{\mathcal{HC}}$ contains a conic occur when $q \equiv
2,3,7,8,12,13 \bmod{15}$ which simplifies to $q \equiv \pm 2
\bmod{5}$. These cases give rise to flocks of quadratic cones.
For $q$ odd, these are equivalent to the Kantor K2 flocks
\cite{JAT:87} and for $q$ even (more precisely, $q = 2^{2k+1}$,
occurring when $q$ is even and  $q \equiv 2,8 \bmod{15}$) they
are known as the Payne \cite{SEP:85} flocks. As the description
of this flock is characteristic-free, we would prefer to call
this flock the Kantor-Payne flock (as has been done elsewhere in
the literature).

    Another example, illustrating a flock whose coordinate functions
are not all monomial, can be obtained by considering the function
$f(x) = x^5 + 2nx^3 + n^2x$ where $n$ is a non-square in $GF(q)$.
Dickson \cite{LED:01} has shown that this is a permutation
polynomial if and only if $q = 5^e$. This implies that over
$GF(5^e)$, $f(x+a) - (2na^3 + n^2a) = x^5 + 2nx^3 + 6nax^2 +
(6na^2 + n^2)x$ is a permutation polynomial $\forall a \in
GF(5^e)$. Consider the flock $\mathcal{F}(t,t^2,t^5 + 2nt^3)$.
For $q = 5^e$, $\mathbf{P}_{\mathcal{HC}}$ of this flock consists
of $\{(1,0,0,0)\} \cup \{(na^2 + n^2, na, 1, 0) \mid a \in
GF(5^e)\}$, in other words, the points of the conic $nx_0x_2 =
{x_1}^2 + n^3{x_2}^2$. These flocks are equivalent to the K3 or
``Kantor likeable'' flocks due to Gevaert and Johnson
\cite{GeJo:88}. Consider the special case of $q = 5$. Since $x^5 =
x$ over $GF(5)$, the permutation polynomial reduces to $x^3 +
3ax^2 + (3a^2 + 3n + 1)x$. This in turn implies that the Kantor
likeable flocks in $PG(3,5)$ are equivalent to the FTW flock, a
point that is implied, but never stated in \cite{DeHe:92}.

    Our last example is restricted to even characteristic. Let $q =
2^e$ and chose $i < e$ so that $(2^i + 1, q-1) = 1$. In this case,
$f(x) = x^{2^i+1}$ is a permutation polynomial over $GF(2^e)$.
Therefore, $f(x+a) - a^{2^i+1} = x^{2^i+1} + ax^{2^i} + a^{2^i}x$
is a permutation polynomial over $GF(q),\, \forall a \in GF(q)$.
The $\mathbf{P}_{\mathcal{HC}}$ of the flock
$\mathcal{F}(t,t^{2^i},t^{2^i+1})$ contains the points of $x_0x_2
= {x_1}^{2^i}$. If $(i,e) = 1$ this curve is a translation oval
and $\mathcal{F}(t,t^{2^i},t^{2^i+1})$ is an $\alpha$-flock
\cite{WEC:98}. This situation arises only when $e$ is odd. This
is an alternate (and simpler) proof of Theorem 5 in
\cite{WEC:98}, whose proof simplified that of a result of Fisher
and Thas \cite{FiTh:79}. When $(i,e) > 1$ we obtain flocks of
non-oval cones of a type that we have called $\beta$-flocks
(\cite{WEC:98b}).

\section{The Classification of Flocks for $q \le 7$}

All flocks of arbitrary cones can be easily determined for small
$q$, by examining herd spaces. Since there exist cones which do
not admit any flock \cite{WEC2}, all references to cones in this
section are to non-empty cones which admit at least one flock.
The material in this section should be compared to \cite{JAT:87}
and \cite{DeGeTh:88} where the flocks of quadratic cones for $q$
in this range are determined by other methods. \\

\noindent \textbf{Remark:} We are not attempting a complete
classification of these flocks, although the methods used here
could be used to do so. Rather, we are trying to classify the
flocks of ``interesting'' cones, that is, cones whose carriers
are not too small and/or contained in just a few lines. While it
is possible to define ``interesting'', such a definition is bound
to be arbitrary in nature and so we will not do so here. Also, we
will not try to refine the classification of the star flocks other
than to indicate when they are linear. Star flocks are examined
in more detail in \cite{WEC2}.

\subsection{$q=2$}
Since any two distinct planes of $PG(3,q)$ meet in a line, a
flock of any cone in $PG(3,2)$ is a linear flock.

\subsection{$q=3$}

$\mathbf{P}(\mathcal{Z})$ is isomorphic to $PG(1,3)$, the
projective line, and $\mathcal{S}$ is a point of this line.
Obviously, all flocks in $PG(3,3)$ are star flocks. The point set
of the herd cover of any proper star herd space consists of just
one line by Proposition \ref{P:pshs}. Therefore, if the carrier
of a cone contains at least three non-collinear points, a flock
of the cone must be linear. Thus the flocks of all quadratic
cones in $PG(3,3)$ are linear.

\subsection{$q=4$}
$\mathbf{P}(\mathcal{Z})$ is isomorphic to $PG(2,4)$, the
projective plane of order 4, and $\mathcal{S}$ is a pair of points
in this plane. If a cone has more than two points in its carrier,
then it can only admit star flocks. A herd cover of a proper star
herd space can contain at most two lines, so if the carrier of a
cone contains a 5-arc, the cone can admit only linear flocks. The
only non-star flock is the FTW flock, but its critical cone is
flat. We again have that the flocks of all quadratic cones in
$PG(3,4)$ are linear.

\subsection{$q=5$}

$\mathbf{P}(\mathcal{Z})$ is isomorphic to $PG(3,5)$ and
$\mathcal{S}$ consists of 6 points which lie in a hyperplane,
i.e., they are coplanar in this space. The degree of a
permutation polynomial over $GF(5)$ can only be 1 or 3. The six
permutation points are $\langle t \rangle$ and the five points
$\langle 3at + 3a^2t^2 + t^3 \rangle$ for $a \in GF(5)$. The
point set of the herd cover is the conic $2x_0x_2 = {x_1}^2$ in
the herd space $\Gamma(t,t^2,t^3)$. Any other plane of
$\mathbf{P}(\mathcal{Z})$ can contain at most two permutation
points. Thus, the only non-star flock of a cone of $PG(3,5)$
whose carrier contains at least three points is the FTW flock (up
to equivalence). There are at most two lines in the point set of
the herd cover of any proper star herd space. A star flock of a
cone in $PG(3,5)$ whose carrier contains a 5-arc must be a linear
flock. The flocks of quadratic cones are either linear or FTW.

\subsection{$q=7$}

In this case $\mathbf{P}(\mathcal{Z})$ is isomorphic to $PG(5,7)$
and contains $5! =120$ permutation points. The permutation
polynomials over $GF(7)$ can only have degrees of 5, 4 or 1.
Using Dickson's list, we can explicitly exhibit all the
permutation polynomials over $GF(7)$. This is done in Table
\ref{Ta:three}.
\begin{table}
 \begin{center}
  \begin{tabular}{| c | c | l |}
   \hline
   \emph{Type} & \emph{Num.} & \emph{Permutation Polynomial} \\ \hline \hline
    I  &  1  &   $t$ \\ \hline
    II &  49 &{\parbox{4in}{ $(5a^4 + 3a^2b + 3b^2)t + (3a^3 + 3ab)t^2 + (3a^2 +
    b)t^3\\ \qquad + 5at^4 + t^5, \,\forall a,b\in GF(7)$ }}\\ \hline
    III & 52 & {\parbox{4in}{$(5a^4 + 3a^2n \pm 2a + 3n^2)t + (3a^3+3an \pm 1)t^2
    +(3a^2+n)t^3\\ \qquad + 5at^4 +t^5, \, \forall a \in GF(7), n$ a
    nonsquare }}\\ \hline
    IV & 14 & $(5a^4 \pm 4a)t + (3a^3 \pm 2)t^2 + 3a^2t^3 + 5at^4
    + t^5, \forall a \in GF(7)$ \\ \hline
    V & 14 & $(4a^3 \pm 3)t + 6a^2t^2 + 4at^3 + t^4, \forall a \in
    GF(7)$ \\ \hline
  \end{tabular}
  \caption{Representative Permutation Polynomials over $GF(7)$}\label{Ta:three}
  \end{center}
\end{table}

As points of $\mathbf{P}(\mathcal{Z})$ these 120 permutation
points lie in the hyperplane $\mathbf{P}(\mathcal{Z}_1) =\langle
t,t^2,t^3,t^4,t^5 \rangle$. Any plane of
$\mathbf{P}(\mathcal{Z})$ which does not lie in this hyperplane
must meet the hyperplane in a line. The permutation points of the
non-star herd spaces that correspond to these planes must all lie
on a line, and so, such flocks can only have flat cones as
critical cones. A non-star flock of a non-flat cone would
therefore correspond to a plane which lies in
$\mathbf{P}(\mathcal{Z}_1)$. A simple calculation shows that all
the permutation points lie on the quadric $Q(4,7)$ of
$\mathbf{P}(\mathcal{Z}_1)$ given by $3x_2^2 = x_0x_4 + x_1x_3,\,
x_5 = 0$.  At each point of the hyperbolic quadric, there are 8
generator lines which form a quadratic cone that lies in a unique
3-space whose only intersection with the hyperbolic quadric is
this cone. For example, at the point $\langle t \rangle$ on the
quadric, the type V points lie on seven generators of such a cone
with vertex $\langle t \rangle$, and the 3-space containing this
cone is $\langle t, t^2, t^3, t^4 \rangle$. The type II points
lie in seven planes, $\pi_a : \langle t, (t+a)^3 -a^3, (t+a)^5
-a^5 \rangle, a \in GF(q)$, the points in each plane together with
$\langle t \rangle$ forming a conic. The remaining points come in
plus/minus pairs, and each such pair is collinear with a unique
point of type II (those with $b = 0$ or $b$ a non-square in
$GF(q)$). These lines are generator lines of the hyperbolic
quadric (lying completely in the quadric). To determine all
possible non-star flocks of non-empty cones we need only examine
the planes of $\mathbf{P}(\mathcal{Z}_1)$ which contain $\langle
t \rangle$. There are only four types of planes, with respect to
the hyperbolic quadric, through a point of the quadric. Namely,
planes whose intersection with the hyperbolic quadric consists of
precisely 1 generator, 2 generators, a conic or a single point.
Thus, the herd cover of any non-star flock that is not contained
in two lines must be an arc which lies on a conic. A simple
computer-aided calculation shows that the only planes through
$\langle t \rangle$ which contain at least seven points are the
$\pi_a$. The flocks that correspond to these planes are all
projectively equivalent, as the following calculation shows:
  \begin{equation*}
    \left(
    \begin{matrix}
    1 & 0 & 0 & -a \\
    0 & 1 & 0 & -a^3 \\
    0 & 0 & 1 & -a^5 \\
    0 & 0 & 0 & 1
    \end{matrix}
    \right) \left(
    \begin{matrix}
    t \\
    t^3\\
    t^5\\
    1
    \end{matrix}
    \right) = \left(
    \begin{matrix}
     t - a \\
    t^3 - a^3\\
    t^5 - a^5\\
    1
    \end{matrix}
    \right) = \left(
    \begin{matrix}
     s \\
    (s+a)^3 - a^3\\
    (s+a)^5 - a^5\\
    1
    \end{matrix}
    \right) ,
  \end{equation*}
\noindent where the last step is just a reparameterization of the
flock by $s = t-a$. We can therefore conclude that the only
non-star flocks of quadratic cones in $PG(3,7)$ are equivalent to
$\mathcal{F}(t,t^3,t^5)$, the Kantor-Payne Flock.

    A line can intersect $\mathcal{S}$ in at most 3 points (a
generator of the hyperbolic quadric). So, there are at most three
lines in the point set of the herd cover of any proper star herd
space. A star flock of a cone in $PG(3,7)$ whose carrier contains
a 7-arc must be a linear flock. The flocks of quadratic cones are
either linear or Kantor-Payne.

\noindent
{\bf Address of the author:}\\
W.E. Cherowitzo \\
Department of Mathematical and Statistical Sciences\\
University of Colorado Denver\\
Campus Box 170, P.O. Box 173364\\
Denver, CO 80217-3364\\
U.S.A.\\
e-mail: william.cherowitzo@ucdenver.edu\\
http://math.ucdenver.edu/$\sim$wcherowi\\

\end{document}